\documentclass[11pt,a4paper]{article}
\usepackage{amssymb}
\usepackage{graphicx}
\usepackage{caption}
\usepackage{amsmath}
\usepackage{xcolor}
\newtheorem{theorem}{Theorem}

\newtheorem{corollary}{Corollary}

\newtheorem{definition}{Definition}

\newtheorem{notation}{Notation}

\newtheorem{remark}{Remark}

\newenvironment{proof}[1][Proof]{\textbf{#1.} } {\ \ \hfill\hbox to .1pt{} \hfill\hbox to .1pt{}
\hfill$\blacksquare$\par}
\begin{document}
\centerline{\Large{\bf{Global rational stabilization of a class of nonlinear time-delay systems}}}
\centerline{}
\centerline{\bf {Nadhem  ECHI\textsuperscript{a}, Boulbaba GHANMI\textsuperscript{b}}}
\centerline{}

\centerline{\textsuperscript{a,b}Gafsa University, Faculty of Sciences of Gafsa}
\centerline{Department of Mathematics, Zarroug   Gafsa 2112 Tunisia}
\centerline{\textsuperscript{a}E-mail: nadhemechi\_fsg@yahoo.fr}

\abstract{The present paper is mainly aimed at introducing a novel notion of stability of nonlinear time-delay systems called Rational Stability.
According to the Lyapunov-type, various sufficient conditions for rational stability are reached.
Under delay dependent conditions, we suggest a nonlinear time-delay observer to estimate the system states, a state feedback controller
 and the observer-based controller  rational stability is provided. Moreover, global rational stability using output feedback is given.
Finally, the study presents simulation findings to show the feasibility of the suggested strategy.
}
 \bigskip

 \vspace*{0.1cm}
 {\bf Mathematics Subject Classification.} 93C10, 93D15, 93D20.

{\bf Keywords.} Rational stability; delay system; nonlinear observer; Lyapunov functional.
\vspace*{0.1cm}
\section{Introduction}
Time-Delay Systems (TDSs)  is also known as call  systems with aftereffect or dead-time, hereditary systems, equations with deviating argument,
or differential-difference equations. They are part of the class of functional differential equations which are infinite-dimensional,
as opposed to ordinary differential equations (ODEs). Time-delay has a number of characteristics. It appears in several control systems, including aircraft,
 chemical \cite{mounier}, biological systems \cite{lili},
engineering, electrical \cite{anthonis}, economic model \cite{hamed, ghanes},
or process control systems, and communication networks, either in the state, the control input, or the measurements \cite{Baillieul, Natori}.
There, we can find transported, communication, or measurement delays. It is noticeable that time delay can cause different problems, such as instability,
divergence behavior,
and oscillation of dynamic systems.
A considerable amount of studies have analyzed the stability of dynamic systems with a delay.
Therefore, the stability of systems with time delay has been investigated extensively over the past decades.
It is a well known fact that stability of nonlinear time-delay systems in Lyapunov sense plays a major role
in control theory,
 and becomes a challenging problem both in theory and applications.
The stability analysis of time delay systems has been recently studied in many areas. There are two crucial kinds of stability of dynamical systems.
 These include asymptotic stability and exponential stability.
In the case of asymptotic stability, for more details, the reader is referred to
  \cite{Be, nadhem2017, Ibrir, germani2000, germani2001, sun, tsinias} and references therein.
\cite{Ramasamy2016} addressed the problem of asymptotic stability for Markovian jump that generalized neural networks with interval time-varying delay systems.
Based on the Lyapunov method, it is suggested that asymptotic stability can be used to solve linear matrix inequality with triple integral terms a delay.
For exponential stability, it is requires that all solutions starting near
 an equilibrium point not only stay nearby, but tend to the equilibrium point very fast with exponential decay rate;
  see \cite{Thun, Benad, Rajchakit, phat}.

 A new notion of stability known as rational stability for systems without time delays is introduced in \cite{hahn}.
The study demonstrates the characteristics of rational stability. It can be characterized by means of Lyapunov functions.
This notion did not know any intense progress like other tools of the stability theory.
For free-delay system, \cite{jammzi2013} studied the issue of  rational stability of continuous autonomous systems,
followed by several examples of control systems.
Under a Hamilton-Jacobi- Belleman approach, some sufficient conditions are developed by \cite{zaghdoudi} to achieve the rational stability of optimal control
for every dynamical control systems.

The questions which are worth being raised here are can we speak about rational stability for time-delay systems?
What is the advantage of this stability? The aim of the current study is to present a new term of stability for nonlinear time-delay systems.
This term is called rational stability.
Sometimes the decay of the energy or the Lyapunov function is not exponential, but it can be polynomial.
 For rational stability and especially when the Jacobean matrix is no longer Hurwitz and the transcendental characteristic polynomial,
 the solutions do not decrease exponentially.
  However, in some cases, the solutions decrease like $t^{-r},\,r>0$ with $r$ is called the rate decay of the solution.
  The real $r$ measures the velocity of convergence of the solution which is crucial in several practical engineering such as satellite systems,
  unicycle systems, underwater, transport equation, string networks, etc.

The current paper introduces a novel notion of stability of nonlinear time-delay systems called rational stability.
It also investigates the problem of output feedback stabilization of a class of nonlinear time
delay system written in triangular form, with constant delay. We impose a generalized
 condition on the nonlinearity to cover the time-delay systems is considered by  \cite{Ibrir}.
Motivated by \cite{jammzi2013} and \cite{zaghdoudi},  Lyapunov-Krasovskii functional is used for the purpose of obtaining
to establish globally rational stability of the closed loop systems.
We  design a nonlinear observer to estimate the system states.
Then, it is used to obtain a new state and input delay-dependent criterion that ensures the rational stability of
the closed-loop system with a state feedback controller. The global rational stability using output feedback is also presented.

The rest of this paper is organized as follows. The
next section presents the definition of rational stability and an auxiliary result concerning a functional should satisfy for guaranteeing the
rational stability. In section 3, It also shows the system description.
The main results are stated in section 4, it is concluded that parameter dependent linear state and output feedback controllers are synthesized
to ensure global rational stability of the nonlinear time delay system. In section 5, we establish the problem of global rational
stability using output feedback.
Finally, an illustrative example, of network-based control systems (NBCSs),
is discussed to demonstrate the effectiveness of the obtained results.
\section{Definitions and auxiliary results}
  Consider  time delay system of the form:
\begin{equation}\left\{
                   \begin{array}{ll}
                     \dot{x}(t)=f(x(t),x(t-\tau)) & \hbox{} \\
                     x(\theta)=\varphi(\theta) & \hbox{ }\\
                   \end{array}
                 \right.
\label{1}
 \end{equation}
where $\tau > 0$ denotes the time delay. The knowledge of $x$ at time $t = 0$ does not allow to
deduce $x$ at time $t$. Thus, the initial condition is specified as a continuous function
$ \varphi\in\mathcal{C}$, where $\mathcal{C}$ denotes the Banach space
of continuous functions mapping the interval $[-\tau, 0]\rightarrow\mathbb{R}^{n}$ equipped with the supremum-norm:
$$\parallel \varphi\parallel_{\infty}\ =\ \max_{\theta\in[-\tau,0]}\parallel\varphi(\theta)\parallel$$
$\| \ \|$ being the Euclidean-norm. The map $f :\mathbb{R}^{n} \times \mathbb{R}^{n}\rightarrow \mathbb{R}^{n}$
is a  locally Lipschitz function, and satisfies
$f(0, 0) = 0.$

  The function segment $x_{t}$ is defined by $x_{t}(\theta) = x(t + \theta),\,\ \theta \in[-\tau, 0].$
For $\varphi \in \mathcal{C}$, we
denote by $x(t,\varphi)$ or shortly $x(t)$ the solution of \eqref{1} that satisfies $x_{0} =\varphi. $ The segment
of this solution is denoted by $x_{t}(\varphi)$ or shortly $x_{t}$.

Inspired from \cite{jammzi2013} and \cite{hahn}, we introduce some definition of rational stability for the time-delay systems.
 \begin{definition}
 The zero solution of \eqref{1} is called
 \begin{itemize}
   \item Stable, if for any $\varepsilon > 0$ there exists $\delta > 0$ such that
$$\|\varphi  \|_{\infty} <\delta \Rightarrow \|x(t)\| < \varepsilon,\,\ \forall t \geq 0.$$
   \item  Rationally stable, if it is stable and there exist positive numbers $M,\,  k,\, r,\, e\leq1$  such that if
   \begin{equation}
\|\varphi  \|_{\infty} <\sigma \Rightarrow\|x(t)\|\leq\frac{M\|\varphi\|_{\infty}^{e}}{(1+\|\varphi\|_{\infty}^{k}t)^{\frac{1}{k}}},\,\ \forall t \geq 0.
    \label{rat}\end{equation}
   \item Globally rationally stable, if it is stable and $\delta$ can be chosen arbitrarily large
for sufficiently large $\varepsilon$, and \eqref{rat} is satisfied for all $\sigma > 0$.
 \end{itemize}
    \end{definition}

Sufficient conditions for rotational stability of a functional differential equation are provided by \cite{jammzi2013},
a generalization of time delay system given by following theorem. For a locally Lipschitz functional
$V : \mathcal{C} \rightarrow \mathbb{R}_{+}$, the derivative of V along the solutions of \eqref{1} is defined as
$$\dot{V}= \lim _{h\rightarrow0}\frac{1}{h}(V (x_{t+h}) - V (x_{t})).$$
\begin{remark}
It is easy to see that rational stability is satisfied then asymptotic stability is satisfied, but the converse is not true.
\end{remark}
    \begin{theorem}\label{th1} Assume that there exist positive numbers $\lambda_{1} ,\, \lambda_{2},\, \lambda_{3} ,\, r_{1} ,\, r_{2} ,\, k$
and a continuous differentiable functional $V:\mathcal{C}\rightarrow\mathbb{R}_{+}$ such that:
\begin{eqnarray}\lambda_{1}\parallel x(t)\parallel^{r_{1}} \leq V(x_{t})&\leq&
  \lambda_{2}\parallel x_{t}\parallel_{\infty}^{r_{2}}, \label{i}\\
  \dot{V}(x_{t})+\lambda_{3} V^{1+k}(x_{t})&\leq&0, \label{ii}\end{eqnarray}
  then, the zero solution of \eqref{1} is globally rationally.
  \end{theorem}
  \begin{proof}
  Using  \eqref{ii}, we have for $V\neq0$, $$\frac{d}{d\theta}V^{-k}(x_{\theta})\geq k\lambda_{3}$$
  Integrating between $0$ and $t$, one obtains $$\int_{0}^{t}\frac{d}{d\theta}V^{-k}(x_{\theta})d\theta\geq \int_{0}^{t}k\lambda_{3}d\theta$$
  equivalently, for all $t\geq 0$
  $$V^{k}(x_{t})\leq \frac{1}{k\lambda_{3}t+V^{-k}(\varphi)}.$$
  Now, it follows Theorem \ref{th1}, condition \eqref{i} that
  $$\|x(t)\|\leq (\frac{1}{\lambda_{1}})^{\frac{1}{r_{1}}}
  \frac{1} {(\lambda_{2}^{-k}\|\varphi\|^{-r_{2}k}_{\infty}+\lambda_{3}k t)^{\frac{1}{k r_{1}}}}.$$
  \end{proof}
 \begin{corollary}\label{coro1} Assume that there exist positive numbers
  $\lambda_{1} ,\, \lambda_{2},\, \lambda_{3} ,\, r_{1} ,\, r_{2} ,\, r_{3},\, r_{2}<r_{3}$
and a continuous differentiable functional $V:\mathcal{C}\rightarrow\mathbb{R}_{+}$ such that:
\begin{eqnarray}
\lambda_{1}\parallel x(t)\parallel^{r_{1}} \leq V(x_{t})&\leq&
  \lambda_{2}\parallel x_{t}\parallel_{\infty}^{r_{2}}, \nonumber\\
\dot{V}(x_{t})&\leq&-\lambda_{3} \|x_{t}\|_{\infty}^{r_{3}},\label{i2}
\end{eqnarray}
  then, the zero solution of \eqref{1} is globally rationally stable.
  \end{corollary}
  \begin{proof}
  The conditions \eqref{i} and \eqref{i2} imply that zero solution of \eqref{1} is stable.\\
  By combining the assertions \eqref{i} and \eqref{i2}, we obtain
  \begin{equation}\dot{V}(x_{t}) \leq -\frac{\lambda_{3}}{\lambda_{2}^{\frac{r_{3}}{r_{2}}}}V^{\frac{r_{3}}{r_{2}}}(x_{t})  \label{v1}\end{equation}
 equivalent to
 \begin{equation*}\dot{V}(x_{t}) \leq -\frac{\lambda_{3}}{\lambda_{2}^{\frac{r_{3}}{r_{2}}}} V^{1+k}(x_{t}) \end{equation*}
 where $k=\frac{r_{3}-r_{2}}{r_{2}}$.\\
Hence, from Theorem \ref{th1}, the zero solution of \eqref{1} is globally rationally stable.
  \end{proof}
  Let us recall here that a function
 $\alpha : \mathbb{R}_{+} \rightarrow \mathbb{R}_{+}$ is of class $\mathcal{K}$ if it is continuous, increasing
and $\alpha(0) = 0$, of class $\mathcal{K}_{\infty}$ if it is of class $\mathcal{K}$ and it is unbounded. The following theorem
provides sufficient Lyapunov-Krasovskii conditions for global rationally stability of the
zero solution of system \eqref{1}.
  \begin{theorem} Assume that there exist positive numbers $\lambda_{1} ,\, \lambda_{2},\, r_{1} ,\, r_{2} ,\, k,$
  $\alpha$ a function of class $\mathcal{K}$
and a continuous differentiable functional $V:\mathcal{C}\rightarrow\mathbb{R}_{+}$ such that:
\begin{description}
  \item[(i)]
$\lambda_{1}\parallel x(t)\parallel^{r_{1}} \leq V(x_{t})\leq
  \lambda_{2}\parallel x_{t}\parallel_{\infty}^{r_{2}},$
 \item[(ii)] $\dot{V}(x_{t})\leq-\alpha(V(x_{t})),\label{i3}$
 \item[(iii)]  $\displaystyle\lim_{t\rightarrow0}\frac{\alpha(t)}{t^{k+1}}=l\in]0, +\infty].\label{i4}$
\end{description}
  Then the zero solution of \eqref{1} is globally rationally stable.
  \label{th2}\end{theorem}
    \begin{proof}
  \underline{First case:} $0 < l < +\infty$. The conditions \eqref{i} and \eqref{i3} imply that zero solution of system \eqref{1} is
stable and attractive, then asymptotically stable and $\displaystyle\lim_{t\rightarrow+\infty}V(x_{t})=0$.\\
Therefore, by using limit definition, there exists $t_{0} > 0$ such
that for every $0 \leq t \leq t_{0}$, we have $\alpha(t)\geq \frac{l}{2}t^{k+1}$.
Since, $\displaystyle\lim_{t\rightarrow+\infty}V(x_{t})=0$, for this $t_{0} > 0$,
there exists $t_{*} >0$ such that for every $t \geq t_{*}$ one gets $0\leq V(x_{t}) \leq t_{0}$ and $$\alpha(V(x_{t}))\geq \frac{l}{2}V^{1+k}(x_{t}).$$
Using \eqref{i3}, we obtain
$$0\geq \dot{V}(x_{t})+\alpha V(x_{t})\geq\dot{V}(x_{t})+\frac{l}{2}V(x_{t})^{k+1}.$$
Thus, we have $$\dot{V}(x_{t})\leq\frac{l}{2}V(x_{t})^{k+1}.$$
Then, using the  Theorem \ref{th1}, we can conclude
that the zero solution of \eqref{1} is globally rationally stable.\\
\underline{Second case:} $l = +\infty$. As in the proof of first case, there exists $t_{0} > 0$ such
that for every $0 \leq t \leq t_{0}$, we have $$\alpha(t)\geq  t^{k+1}\ and \ \dot{V}(x_{t})\leq V(x_{t})^{k+1}.$$
  \end{proof}
\begin{remark}
The Theorem \ref{th1}, Theorem \ref{th2} and Corollary \ref{coro1} generalize the results given by \cite{jammzi2013} for the case of free-delay system.
\end{remark}
   \section{System description}
  Consider the nonlinear time-delay system:
   \begin{equation}\left\{
                   \begin{array}{l}
                     \dot{x}(t)=Ax(t)+Bu(t) +f(x(t),x(t-\tau),u(t))  \\
                     y(t)=Cx(t)
                   \end{array}
                 \right.
 \label{3}\end{equation}
where $  x\in\mathbb{R}^{n}$ is the state vector, $u \in \mathbb{R}$ is the
 input of the system, $y\in \mathbb{R}$
is the measured output and  $\tau$ is a positive known scalar that denotes the time delay affecting the state variables.
 The matrices $A$, $B$ and $C$ are given by
 $$A=\left[
 \begin{array}{ccccc}
  0 & 1 &0& \cdots & 0 \\
  0 & 0 & 1 & \cdots&0 \\
 \vdots & \vdots& \vdots & \ddots & \vdots \\
 0 & 0 & 0&\cdots & 1 \\
 0 & 0 & 0&\cdots & 0 \\
 \end{array}
 \right],\ \ \,\ B=\left[
 \begin{array}{c}
 0 \\
 0 \\
 \vdots\\
 0 \\
 1\\
 \end{array}
 \right]
,\, C=\left[
    \begin{array}{ccccc}
      1 & 0 & \cdots&0 & 0 \\
    \end{array}
  \right],$$
  and the perturbed term is
$$f(x(t), x(t -\tau), u(t)) =\left[
 \begin{array}{c}
 f_{1}(x_{1}(t),x_{1}(t -\tau), u(t)) \\
 f_{2}(x_{1}(t),x_{2}(t),x_{1}(t -\tau),x_{2}(t -\tau), u(t)) \\
 \vdots \\
f_{n}(x(t),x(t -\tau), u(t))
\end{array}\right].$$
 The mappings
 $f_{i}:\mathbb{R}^{n}\times\mathbb{R}^{n}\times \mathbb{R}\rightarrow\mathbb{R},$ $i=1,\ldots,n,$ are smooth
 and satisfy the following assumption:\\
 We suppose that $f$ satisfies the following assumption:

{\textbf{Assumption 1.}} The nonlinearity $f (y, z,u)$ is smooth,
globally Lipschitz with
respect to $y$ and $z$, uniformly with respect to $u$ and well-defined for all $y,z\in \mathbb{R}^{n}$ with $f (0, 0,u)= 0$.
\\We suppose also that,

{\textbf{Assumption 2.}} For all $t \geq 0$, the delay $\tau$ is known and constant.\\
  \begin{notation}
Throughout the paper, the time argument
is omitted and the delayed state vector $x(t-\tau)$ is noted by $x^{\tau}$. $A^{T}$ means the
transpose of $A$.
$\lambda_{max}(A)$ and $\lambda_{min}(A)$ denote the maximal and minimal eigenvalue of a matrix
 $A$ respectively.  $I$ is an appropriately dimensioned identity matrix,
$diag[\cdots]$ denotes a block-diagonal matrix.  \end{notation}
 \begin{remark}
 This paper  focuses on state observer design for a class of system given by \eqref{3}.
It specifically shows that the general high-gain observer design
framework established in \cite{Gauthier} and
\cite{Hammouri} for free delay systems can be properly extended to this class of time-delay systems.
 \end{remark}
\section{Separation principle}
\subsection{Observer design}
 The observer synthesis for triangular nonlinear system design problems along with time-delay systems have become the focal focus of various studies
  \cite{Be}, \cite{Farza} and \cite{ghanes}
 and references therein.
 Under the global Lipschitz condition, an observer for a class of time-delay nonlinear systems in the strictly lower
 triangular form was proposed by \cite{ghanes}.
In \cite{Be}  a nonlinear observer is used to investigate the output feedback controller problem for a class of nonlinear
delay systems for the purpose of calculating the system states.
  Based on time-varying delays  known and bounded, \cite{Farza} propose a nonlinear observer   for a class of time-delay nonlinear systems.
In this section we  devote to the design of the observer-based
controller. \begin{equation}\label{obs}
\dot{\hat{x}}(t) = A\hat{x} + Bu(t)+f(\hat{x},\hat{x}^{\tau},u)  + L(\theta)( C\hat{x}-y)\end{equation}
where $L(\theta)= [l_{1}\theta,\ldots,l_{n}\theta^{n}]^{T}$ with $\theta>0$ and where
$L= [l_{1},\ldots,l_{n}]^{T}$ is selected such that
$A_{L}:=A+LC$ is Hurwitz, $\hat{x}(s)=\hat{\phi}(s),\ -\tau\leq s\leq0$ with $\hat{\phi}:[-\tau,0]\rightarrow\mathbb{R}^{n}$ being
any known continuous function.
 Let $P$ be the symmetric positive definite solution of the Lyapunov equation
\begin{equation} A^{T}_{L}P + PA_{L} = -I.\label{Ly1}\end{equation}
\begin{theorem}  Suppose that Assumption 1-2 is satisfied and  there
exist  positive constant $\theta$
such that
\begin{equation}\left\{
                  \begin{array}{ll}
                    \frac{\theta}{2}-\|P\|\frac{\ln\theta}{2\tau}-3k\|P\|>0 & \hbox{,} \\\\
                    \frac{\sqrt{\theta}}{2}-k\|P\|\ \ \ \ \ \ \ \ \ \ \ \  >0 & \hbox{.}
                  \end{array}
                \right.
\label{condi1}\end{equation}
Then, \eqref{obs} is globally rationally observer for system \eqref{3}.
\label{th3}\end{theorem}
\begin{proof} Denote $e=\hat{x}-x$ the observation error. We have
 \begin{equation}\dot{e}=(A+L(\theta)C)e+f(\hat{x},\hat{x}^{\tau},u)-f(x,x^{\tau},u)\label{erro}\end{equation}
 For $\theta>0$, let $\Delta_{\theta}=diag[1,\frac{1}{\theta},\ldots,\frac{1}{\theta^{n-1}}]$.
 One can easily check the following identities: $\Delta_{\theta}A\Delta_{\theta}^{-1}=\theta A,\, C\Delta_{\theta}^{-1}=C$.
 Let us now introduce $\eta=\Delta_{\theta}e$, then we get
 \begin{equation}\dot{\eta}=\theta A_{L}\eta+\Delta_{\theta}(f(\hat{x},\hat{x}^{\tau},u)-f(x,x^{\tau},u))\label{erro1}\end{equation}
 Let us choose a Lyapunov-Krasovskii functional candidate as follows \begin{equation}
 V(\eta_{t})= V_{1}(\eta_{t})+ V_{2}(\eta_{t})\label{veta}\end{equation}
 with $$V_{1}(\eta_{t})=\eta^{T} P\eta$$ and
$$V_{2}(\eta_{t})= \frac{\theta}{2}\theta^{\frac{-t}{2\tau}}\displaystyle\int_{t-\tau}^{t}\theta^{\frac{s}{2\tau}}\|\eta(s)\|^{2}ds.$$
 Since $P$ is symmetric positive definite then for all $\eta\in\mathbb{R}^{n},$
   \begin{equation}\lambda_{\min}(P)\|\eta\|^{2}\leq \eta^{T}P\eta\leq \lambda_{\max}(P)\|\eta\|^{2}\label{norm}\end{equation}
   This implies that on the one hand, $$V (\eta_{t}) \geq \lambda_{\min}(P)\parallel \eta(t)\parallel^{2},$$
   and on the other hand,
$$ \begin{array}{lll}V(\eta_{t})&=&
 \eta^{T} P \eta+\frac{\theta}{2}\displaystyle\int_{-\tau}^{0}\theta^{\frac{\mu}{2\tau}}\parallel \eta(\mu+t)\parallel^{2}d\mu\\
&=&\eta^{T} P \eta+\frac{\theta}{2}\displaystyle\int_{-\tau}^{0}\theta^{\frac{\mu}{2\tau}}\parallel \eta_{t}( \mu)\parallel^{2}d\mu\\
&\leq& \lambda_{\max}(P)\parallel \eta \parallel^{2}+\frac{\theta}{2}
\displaystyle\int_{-\tau}^{0}\theta^{\frac{\mu}{2\tau}}\parallel \eta_{t}\parallel_{\infty}^{2}d\mu\\
&\leq& (\lambda_{\max}(P)+\frac{\theta\tau}{2})\|\eta_{t}\|_{\infty}^{2}.
\end{array}$$
Thus condition $(i)$ of Theorem \ref{th2} is satisfied with $$\lambda_{1}=\lambda_{\min}(P),\
\lambda_{2}= \lambda_{\max}(P)+\frac{\theta\tau}{2},\ r_{1}=r_{2}=2.$$
The time derivative of $V_{1}(\eta_{t} )$ along the trajectories of system \eqref{erro} is
\begin{equation}\dot{V}_{1} (\eta_{t})=  \eta^{T} ( A^{T}_{L}P + PA_{L} )\eta+
 2\eta^{T} P(f(\hat{x},\hat{x}^{\tau},u)-f(x,x^{\tau}u))\label{dv1}.  \end{equation}
The time derivative of $V_{2}(\eta_{t} )$ along the trajectories of system \eqref{erro} is
\begin{equation}
    \dot{V}_{2} (\eta_{t})=\frac{\theta}{2}\| \eta\|^{2} -\frac{\sqrt{\theta}}{2}\|\eta^{\tau}\|^{2} -\frac{\ln\theta}{2 \tau}V_{2}(\eta_{t})\label{dv2}.
    \end{equation}
Next, the time derivative of \eqref{veta} along the trajectories of
system \eqref{erro} and making use of \eqref{Ly1},\eqref{norm}, \eqref{dv1} and \eqref{dv2}, we have
\begin{equation}\begin{array}{lll}
\dot{V} (\eta_{t})
&\leq& -\frac{\theta}{2} \|\eta\|+
2\|\eta \|\| P\|\|\Delta_{\theta}(f(\hat{x},\hat{x}^{\tau},u)-f(x,x^{\tau}u))\|\\\\
& &-\frac{\sqrt{\theta}}{2}\|\eta^{\tau}\|^{2}
-\frac{\ln\theta}{2\tau}V_{2}(\eta_{t}).\end{array}\label{dotv}\end{equation}
Using  \eqref{norm} we obtain
\begin{eqnarray}
\dot{V} (\eta_{t})+\frac{\ln\theta}{ 2\tau}V(\eta_{t})
&\leq& -\left\{\frac{\theta}{2}-\|P\|\frac{\ln\theta}{2 \tau}\right\}\|\eta\|^{2}+
2\|\eta \|\| P\|\|\Delta_{\theta}(f(\hat{x},\hat{x}^{\tau},u)-f(x,x^{\tau}u))\|-\frac{\sqrt{\theta}}{2}\|\eta^{\tau}\|^{2}\nonumber
\label{v+v}\end{eqnarray}
The following inequality hold globally thanks to assumption $\textbf{A1}$ ( as in \cite{Farza, ghanes})
\begin{eqnarray}\|\Delta_{\theta}(f(\hat{x},\hat{x}^{\tau},u)-f(x,x^{\tau}u))\|
&\leq& k_{1}\|\Delta_{\theta}(\hat{x}-x)\|+k_{2} \|\Delta_{\theta}(\hat{x}^{\tau}-x)\|\label{lipsch1}\\ \nonumber\\
&\leq& k\|\eta\|+k \|\eta^{\tau}\|\label{lipsch2}.\end{eqnarray}
where $k_{1},\ k_{2}$ is a Lipschitz constant in \eqref{lipsch1} and $k=\max(k_{1},\ k_{2})$.\\
So, we get
\begin{eqnarray}
\dot{V} (\eta_{t})+\frac{\ln\theta}{ 2\tau}V(\eta_{t})
&\leq& -\left\{\frac{\theta}{2}-\|P\|\frac{\ln\theta}{2 \tau}+2k\|P\|\right\}\|\eta\|^{2}+
2k\| P\|\|\eta\|\|\eta^{\tau}\|^{2}-\frac{\sqrt{\theta}}{2}\|\eta^{\tau}\|^{2}\nonumber
\label{v+v1}\end{eqnarray}
 Using the fact that $$\begin{array}{lll}
2\|\eta\|\|\eta^{\tau}\|&\leq& \|\eta\|^{2}+\|\eta^{\tau}\|^{2},\end{array}$$
we deduce that \begin{eqnarray}
\dot{V} (\eta_{t})+\frac{\ln\theta}{ 2\tau}V(\eta_{t})
&\leq& -\left\{\frac{\theta}{2}-\|P\|\frac{\ln\theta}{2 \tau}-3k\|P\|\right\}\|\eta\|^{2}
-\left\{\frac{\sqrt{\theta}}{2}-k\|P\|\right\}\|\eta^{\tau}\|^{2}\label{v+v-mu}\end{eqnarray}
Let $$ \begin{array}{lll}a(\theta)&=& \frac{\theta}{2}-\|P\|\frac{\ln\theta}{2 \tau}-3k\|P\|\\\\
b(\theta)&=& \frac{\sqrt{\theta}}{2}-k\|P\|\end{array}$$
 Using \eqref{condi1} we have $a(\theta)> 0$ and $b(\theta)> 0$.

Now, the objective is to prove the rational convergence of the observer \eqref{erro}. Inequality \eqref{v+v-mu} becomes
$$\dot{V}(\eta_{t})\leq-\frac{\ln\theta}{2\tau} V(\eta_{t}).$$
Finally, using the  stability Theorem \ref{th2}, we can conclude that the error dynamics \eqref{erro}  is
rationally stable if  \eqref{condi1}  hold.
\end{proof}
\subsection{Global rational stabilization by state feedback}
In this subsection, we establish a delay-dependent condition for the rational state
feedback stabilization of the nonlinear system \eqref{3}. The state feedback controller is given
by \begin{equation} u=K(\theta)x \label{5}\end{equation}
where $K(\theta)= [k_{1}\theta^{n},\ldots,k_{n}\theta]$ and  $K= [k_{1},\ldots,k_{n}]$ is selected such that
$A_{K}:=A+BK$ is Hurwitz.
Let $S$ be the symmetric positive definite solution of the Lyapunov equation
\begin{equation} A^{T}_{K}S + SA_{K} = -I.\label{Ly2}\end{equation}
\begin{theorem}\label{th4} Suppose that Assumption 1-2 is satisfied
and  there exist  positive constant $\theta$ such that
\begin{equation}\left\{
                  \begin{array}{ll}
                    \frac{\theta}{2}-\|S\|\frac{\ln\theta}{2\tau}-3k\|S\|>0 & \hbox{,} \\\\
                    \frac{\sqrt{\theta}}{2}-k\|S\|\ \ \ \ \ \ \ \ \ \ \ \  >0 & \hbox{,}
                  \end{array}
                \right.\label{condi2}\end{equation}
then, the closed loop time-delay system \eqref{3}-\eqref{5} is globally rationally
stable.
\end{theorem}
\begin{proof}
The closed loop system is given by
 \begin{equation}\dot{x}=(A+BK(\theta))x+f(x,x^{\tau},u).\label{dotx} \end{equation}
Let $\chi = \Delta_{\theta}x$. Using the fact that $\Delta_{\theta}BK(\theta) = \theta BK\Delta_{\theta}$ we get
\begin{equation}\dot{\chi}=\theta A_{K}\chi+\Delta_{\theta}f(x,x^{\tau},u).\label{dotchi} \end{equation}
Let us choose a Lyapunov-Krasovskii functional candidate as follows
 \begin{equation} W(\chi_{t})= W_{1}(\chi_{t})+ W_{2}(\chi_{t}),\label{xeta}\end{equation}
 with $$W_{1}(\chi_{t})=\chi^{T} S\chi$$ and
 $$ W_{2}(\chi_{t}) = \frac{\theta}{2}\theta^{\frac{-t}{2\tau}}\displaystyle\int_{t-\tau}^{t}\theta^{\frac{s}{2\tau}}\|\chi(s)\|^{2}ds.$$
 As in the proof of Theorem \ref{th3}, we have
 $$\lambda_{\min}(S)\|\chi(t)\|^{2}\leq W(\chi_{t} )\leq( \lambda_{\max}(S)+\frac{\theta\tau}{2})\| \chi_{t}\|_{\infty}^{2}$$
Thus condition $(i)$ of Theorem \ref{th2} is satisfied with $$\lambda_{1}=\lambda_{\min}(S),\
\lambda_{2}= \lambda_{\max}(S)+\frac{\theta\tau}{2},\ r_{1}=r_{2}=2.$$
The time derivative of \eqref{xeta} along the trajectories of
system \eqref{dotchi} is given by
$$\begin{array}{lll}
\dot{W} (\chi_{t})&=&2\chi^{T}S\dot{\chi}+\frac{\theta}{2}\|\chi\|^{2}-\frac{\sqrt{\theta}}{2}\|\chi^{\tau}\|^{2}-
\frac{\ln\theta}{2\tau}W_{2}(\chi_{t})\\\\
&=& 2\theta\chi^{T}SA_{K}\chi+2\chi^{T}S\Delta_{\theta}f(x,x^{\tau},u)+\frac{\theta}{2}\|\chi\|^{2}
-\frac{\sqrt{\theta}}{2}\|\chi^{\tau}\|^{2}-\frac{\ln\theta}{2\tau}W_{2}(\chi_{t})\\\\
&\leq& -\frac{\theta}{2}\|\chi\|+2\|\chi\|\|S\|\|\Delta_{\theta}f(x,x^{\tau},u)\|
-\frac{\sqrt{\theta}}{2}\|\chi^{\tau}\|^{2}-\frac{\ln\theta}{2\tau}W_{2}(\chi_{t})
.\end{array}$$
 Since $f(0,0,u)=0$, \eqref{lipsch2} implies that
 \begin{equation}\|\Delta_{\theta}f(x,x^{\tau},u)\|\leq k\|\chi\|+k\|\chi^{\tau}\|.\label{delta}\end{equation}
 So
\begin{eqnarray}
\dot{W} (\chi_{t})+\frac{\ln\theta}{2 \tau}W(\chi_{t})
&\leq& -\left\{\frac{\theta}{2} -\|S\|\frac{\ln\theta}{2 \tau}-2k\|S\|\right\}\|\chi\|^{2}+
2k\| S\|\|\chi \|\|\chi^{\tau}\|-\frac{\sqrt{\theta}}{2}\|\chi^{\tau}\|^{2}\label{w+w}\end{eqnarray}
Using the fact that $$\begin{array}{lll}
2\|\chi\|\|\chi^{\tau}\|&\leq& \|\chi\|^{2}+\|\chi^{\tau}\|^{2}.\end{array}$$
We deduce that \begin{eqnarray}
\dot{W} (\chi_{t})+\frac{\ln\theta}{ 2\tau}W(\chi_{t})
&\leq& -\left\{\frac{\theta}{2}-\|S\|\frac{\ln\theta}{2 \tau}-3k\|S\|\right\}\|\chi\|^{2}
-\left\{\frac{\sqrt{\theta}}{2}-k\|S\|\right\}\|\chi^{\tau}\|^{2}.\label{w+w-mu}\end{eqnarray}
Let $$ \begin{array}{lll}c(\theta)&=& \frac{\theta}{2}-\|S\|\frac{\ln\theta}{2 \tau}-3k\|S\|,\\\\
d(\theta)&=& \frac{\sqrt{\theta}}{2}-k\|S\|.\end{array}$$
 Using \eqref{condi2}, we have $c(\theta)> 0$ and $d(\theta)> 0$ which implies that

$$\dot{W}(\chi_{t})\leq-\frac{\ln\theta}{2\tau} W(\chi_{t}).$$
By Theorem \ref{th2}, we conclude that the origin of the closed loop system \eqref{dotx} is globally
rationally stable.
\end{proof}
\subsection{Observer-based control stabilization}
In this subsection, we implement the control law with estimate states. The observer-based
controller is given by:
\begin{equation}u = K(\theta)\hat{x}, \label{cont}\end{equation}
where $\hat{x}$ is provided by the observer \eqref{obs}.
\begin{theorem} Suppose that Assumptions 1-2 are satisfied,
such that conditions \eqref{condi1} and \eqref{condi2} hold. Then the origin of the closed loop time-delay
system \eqref{3}-\eqref{cont} is globally rationally stable.\end{theorem}
\begin{proof}
The closed loop system in the $(\chi, \eta)$ coordinates can be written as follows:
\begin{equation}\begin{array}{lll}\dot{\chi}&=&\theta A_{K}\chi+\theta BK\eta+\Delta_{\theta}f(x,x^{\tau},u),\\
\dot{\eta}&=&A_{L}\eta+\Delta_{\theta}(f(\hat{x},\hat{x}^{\tau},u)-f(x,x^{\tau},u)).\label{w+v}\end{array}\end{equation}
Let $$U(\eta_{t},\chi_{t})=\alpha V(\eta_{t})+ W(\chi_{t}),$$ where $V$ and $W$ are given by \eqref{veta} and \eqref{xeta} respectively.
Using the above results, we get
$$\begin{array}{lll}\dot{U}(\eta_{t},\chi_{t})+\frac{\ln\theta}{2\tau}U( \eta_{t},\chi_{t})&\leq& -\alpha a(\theta)
\|\eta\|^2 - c(\theta)\|\chi\|^{2}+2\theta
\|S\|\|K\|\|\eta\|\|\chi\|.\end{array}$$
Now using the fact that for all $\varepsilon > 0$,
$$2\|\chi\|\|\eta\|\leq\varepsilon\|\chi\|^{2}+\frac{1}{\varepsilon}\|\eta\|^{2},$$
and select $\varepsilon = \frac{c(\theta)}{2\theta\|S\|\|K\|}$, we get
$$\begin{array}{lll}\dot{U}(\eta_{t},\chi_{t})+\frac{\ln\theta}{2\tau}U( \eta_{t},\chi_{t})&\leq& -\alpha a(\theta)
\|\eta\|^2 - \frac{c(\theta)}{2}\|\chi\|^{2}+\frac{2\theta^{2}}{c(\theta)}\|S\|^{2}\|K\|^{2}\|\eta\|^{2}.\end{array}$$
Finally we select $\alpha$ such that $$\alpha a(\theta)-\frac{2\theta^{2}}{c(\theta)}\|S\|^{2}\|K\|^{2}>0,$$
to deduce that the origin of system \eqref{w+v} is globally rationally stable.
\end{proof}
\begin{remark}
It is easy to see that, $a(\theta),\,  b(\theta),\, c(\theta)$ and $d(\theta)$ tend to $\infty$ as $\theta$ tends to $\infty$.
 This implies that there exists $\theta_{0} >1$ such that for all $\theta>\theta_{0}$ conditions \eqref{condi1} and \eqref{condi2} are fulfilled.
\end{remark}
\section{Global rational stabilization by output feedback}
In this subsection, we propose the following system:
\begin{equation}\label{observa}
\dot{\tilde{\hat{x}}}(t) = A\tilde{\hat{x}} + Bu(t)+ L(\theta)( C\tilde{\hat{x}}-y)\end{equation}
The output feedback controller is given by \begin{equation}\label{k2}u=K(\theta)\tilde{\hat{x}}\end{equation}
Under Assumption 1-2, we give now required conditions to ensure that  the origin of system \eqref{3} is rendered globally rationally
stable by the dynamic output feedback control \eqref{observa}-\eqref{k2}.
\begin{theorem}
Consider the time-delay system \eqref{3} under Assumptions 1-2. Suppose
that there exists $\theta > 0$ such that condition \eqref{condi2} holds and
\begin{equation}\frac{\theta}{2}-\|S\|\frac{\ln\theta}{2 \tau}-3k\|S\|>0\label{cond3}\end{equation}
then the closed-loop time-delay system \eqref{3}-\eqref{k2} is globally rationally
stable.
\end{theorem}
\begin{proof} Defining $\widetilde{e}=x-\tilde{\hat{x}}$ the observation error. We have
 \begin{equation}\dot{\widetilde{e}}=(A+L(\theta)C)\widetilde{e}+f(x,x^{\tau},u)\label{errotil}\end{equation}
 For $\theta>0$, let $\Delta_{\theta}=diag[1,\frac{1}{\theta},\ldots,\frac{1}{\theta^{n-1}}]$.
 Let  $\widetilde{\eta}=\Delta_{\theta}\widetilde{e}$, then we get
 \begin{equation}\dot{\widetilde{\eta}}=\theta A_{L}\tilde{\eta}+\Delta_{\theta}f(x,x^{\tau},u)\label{erro11}\end{equation}
 Let us choose a Lyapunov-Krasovskii functional candidate as follows
 \begin{equation} V(\widetilde{\eta}_{t})=\widetilde{\eta}^{T} P\widetilde{\eta}+
 \frac{\theta}{2}\theta^{\frac{-t}{2\tau}}\displaystyle\int_{t-\tau}^{t}\theta^{\frac{s}{2\tau}}\|\widetilde{\eta}(s)\|^{2}ds\label{veta1}\end{equation}
  As in the proof of Theorem \ref{th3}, we have
thus condition $(i)$ of Theorem \ref{th2} is satisfied with $$\lambda_{1}=\lambda_{\min}(P),\
\lambda_{2}= \lambda_{\max}(P)+\frac{\theta\tau}{2},\ r_{1}=r_{2}=2.$$
Following the proof of Theorem \ref{th2}, inequality \eqref{dotv} becomes
  $$\begin{array}{lll}
\dot{V} (\widetilde{\eta}_{t})
&\leq& -\frac{\theta}{2} \|\widetilde{\eta}\|+
2\|\widetilde{\eta} \|\| P\|\|\Delta_{\theta}f(x,x^{\tau},u)\|\\\\
& &-\frac{\sqrt{\theta}}{2}\|\widetilde{\eta}^{\tau}\|^{2}
-\frac{\theta}{2}\frac{\ln\theta}{2\tau}\theta^{\frac{-t}{2\tau}}\displaystyle\int_{t-\tau}^{t}\theta^{\frac{s}{2\tau}}\|\widetilde{\eta}(s)\|^{2}ds.
\end{array}$$
Using  \eqref{norm} and \eqref{delta} we obtain
\begin{eqnarray}
\dot{V} (\widetilde{\eta}_{t})+\frac{\ln\theta}{ 2\tau}V(\widetilde{\eta}_{t})
&\leq& -\left\{\frac{\theta}{2}-\|P\|\frac{\ln\theta}{2 \tau}\right\}\|\widetilde{\eta}\|^{2}+
2k\| P\|\|\widetilde{\eta} \|\|\chi\|+2k\| P\|\|\widetilde{\eta} \|\|\chi^{\tau}\|-\frac{\sqrt{\theta}}{2}\|\widetilde{\eta}^{\tau}\|^{2}\nonumber
\label{v+v}\end{eqnarray}
 Using the fact that $$\begin{array}{lll}
2\|\widetilde{\eta}\|\|\chi\|&\leq& \|\widetilde{\eta}\|^{2}+\|\chi\|^{2}\end{array}$$and
$$\begin{array}{lll}2\|\widetilde{\eta}\|\|\chi^{\tau}\|& \leq& \|\widetilde{\eta}\|^{2}+\|\chi^{\tau}\|^{2},\end{array}$$
we deduce that
\begin{eqnarray}
\dot{V} (\widetilde{\eta}_{t})+\frac{\ln\theta}{ 2\tau}V(\widetilde{\eta}_{t})
&\leq& -\left\{\frac{\theta}{2}-\|P\|\frac{\ln\theta}{2 \tau}-2k\|P\|\right\}\|\widetilde{\eta}\|^{2}+k\|P\|\|\chi\|^{2}
+k\|P\|\|\chi^{\tau}\|^{2}-\frac{\sqrt{\theta}}{2}\|\widetilde{\eta}^{\tau}\|^{2}\label{v+v-mu1}\end{eqnarray}
Let $$U(\widetilde{\eta}_{t},\chi_{t})=\alpha V(\widetilde{\eta}_{t})+ W(\chi_{t}),$$
where  $W$ is given by  \eqref{xeta}. Using \eqref{w+w} and \eqref{v+v-mu1}, we get
 $$\begin{array}{lll}\dot{U}(\widetilde{\eta}_{t},\chi_{t})+\frac{\ln\theta}{2\tau}U( \widetilde{\eta}_{t},\chi_{t})&\leq&
  -\alpha \left\{\frac{\theta}{2}-\|P\|\frac{\ln\theta}{2\tau}-2k\|P\|\right\}\|\tilde{\eta}\|^2 \\\\
& &-  \left\{c(\theta)-\alpha k\|P\|\right\}\|\chi\|^{2}-\left\{d(\theta)-\alpha k\|P\|\right\}\|\chi^{\tau}\|^{2}.\end{array}$$
Finally, we select $\alpha$ such that $$\alpha<\min\left(\frac{c(\theta)}{\alpha k\|P\|},\frac{d(\theta)}{\alpha k\|P\|}\right)$$
to get $$\dot{U}(\widetilde{\eta}_{t},\chi_{t})+\frac{\ln\theta}{2\tau}U( \widetilde{\eta}_{t},\chi_{t})\leq 0$$
Therefore, the closed-loop system is globally rationally stable.
\end{proof}
\begin{remark}
The given controller in Theorem \ref{th3} depends on the nonlinearity $f$ and the time delay
$\tau$, but the controller \eqref{observa}-\eqref{k2} is independent of $f$.
\end{remark}
\begin{remark}
It is worth mentioning that exponential stability including both passivity and dissipativity of generalized
neural networks with mixed time-varying delays are developed in   \cite{Ramasamy2017} by
using the Lyapunov Krasovskii approach in combination with linear matrix inequalities.
These conditions rely on the bounds of the activation functions.
In this paper, we utilize parameter-dependent control laws. We assume that there exists a linear feedback that asserts global rational
 stability of the linear part.
Hence, we select  the $\theta$- parameter in order to establish global rational stability of the nonlinear system under the same controller.\end{remark}
\begin{remark}
\cite{moon,li} show the sufficient conditions which guarantee that the calculation error converges asymptotically
 towards zero in terms of a linear matrix inequality.
As compared to \cite{moon,li}, our results are less conservative and more convenient to use
and  it seems natural and attractive to improve feedbacks and to get solutions decreasing to zero faster.\end{remark}
\begin{remark}
It is worth noting that the obtained findings can be used in multiple time-delays nonlinear systems in the upper-triangular form.
\end{remark}
\section{Simulation results}
This section presents experimental results, in the case of constant delay as an example of practical
application of the time-delay method in actual network-based control systems.
The dynamics of the network-based system are represented by:
\begin{equation}\begin{array}{lll}
\dot{x}_{1}&=&x_{2}(t)+x_{1}\cos x_{1}+x_{1}(t-\tau)\cos u,\\
\dot{x}_{2}&=&u,
\end{array}\label{*}\end{equation}
where $x(t)$ is the augmented state vector containing the plant state vector
and $\tau$ indicates the sensor-to-controller delay in the continuous-time case, is supposed to be constant.
The difference between $\hat{x} ( t )$ and $x ( t )$ is formulated as an error of the network-based system.
The system nonlinearities are globally Lipschitz.
Following the notation used throughout the paper, let
$f_{1}(x,x^{\tau},u)=x_{1}\cos x_{1}+x_{1}(t-\tau)\cos u$,
$f_{2}(x,x^{\tau},u)=0$.
Now, select $L=\left[
            \begin{array}{ccc}
              -14& -28 \\
            \end{array}
          \right]^{T}$ and $K=\left[
            \begin{array}{ccc}
              -30& -30 \\
            \end{array}
          \right]$, $A_{L}$ and $A_{K}$ are Hurwitz.
          Using Matlab, the solutions of the Lyapunov
equations \eqref{Ly1} and \eqref{Ly2} are given respectively by
$$P=\left[
  \begin{array}{ccc}
     0.0377  &  0.0278\\
    0.0278  &  1.0675\\
  \end{array}\right],\ \  S=\left[
  \begin{array}{ccc}
     0.5172 &   -0.5000\\
    -0.5000  &  0.5167\\
  \end{array}\right].$$
 So, $\|P\|=1.0682$ and $\|S\|=1.0169$.
 For our numerical simulation, we choose constant  delay $\tau = 1,$ and the initial conditions for the system are $x(0) = \left[
            \begin{array}{ccc}
              -20& -10 \\
            \end{array}
          \right]^{T} $,
for the observer $\hat{x}(0) = \left[
      \begin{array}{ccc}
     10& 10 \\
\end{array}
\right]^{T}$ and $\theta=8.$
 Corresponding numerical simulation results  are shown in Figures \ref{fig1}-\ref{fig2}.

\section{Conclusion}
In this paper, rational stability and stabilization are investigated for nonlinear time-delay systems.
 Based on this study, we reached a novel result in global rational stability and  stabilization of a class nonlinear time-delay systems.
This class of systems deals with the systems that have a triangular structure.
Based on the result, it was found that the Lyapunov approach was used to perform sufficient conditions for rational stability.
The novel design plays a crucial role in getting a rational stability condition and rendering our approach application to a general
class of systems, namely the class of nonlinear time-delay systems in a lower triangular form.
The numerical result of an example is provided to show the
effectiveness of the proposed approach. As a perspective, It is well known that delay-dependent conditions reveal
less conservative than delay-dependent ones, it can be developed in future research by
 considering other Lyapunov-Krasovskii functional to derive delay dependent conditions.

\end{document}